\documentclass[11pt]{amsart}

\usepackage{amsmath, amssymb, amsthm, mathrsfs, graphicx}
\usepackage{hyperref, enumitem, xspace, ifthen,comment}
\usepackage[all]{xy}
\usepackage{tikz}
\usetikzlibrary{arrows,shapes,trees}

\newtheorem*{thm-plain}{Theorem}
\newtheorem{thm}{Theorem}[section]
\newtheorem{lem}[thm]{Lemma}
\newtheorem{prp}[thm]{Proposition}
\newtheorem{cor}[thm]{Corollary}

\numberwithin{equation}{thm}

\theoremstyle{definition}
\newtheorem{dfn}[thm]{Definition}

\newtheorem*{dfn-plain}{Definition}

\theoremstyle{remark}
\newtheorem{clm}[thm]{Claim}

\newtheorem{rem}[thm]{Remark}

\newtheorem{exm}[thm]{Example}
\newtheorem*{rem-plain}{Remark}

\DeclareMathOperator{\Spec}{Spec}

\DeclareMathOperator{\Aut}{Aut}
\DeclareMathOperator{\Pic}{Pic}

\def\rd#1.{\lfloor{#1}\rfloor}
\def\rp#1.{\lceil{#1}\rceil}

\def\tw#1.{\langle{#1}\rangle}

\newcommand{\lto}{\longrightarrow}

\newcommand{\N}{\mathbb N}
\newcommand{\Z}{\mathbb Z}
\newcommand{\Q}{\ensuremath{\mathbb Q}}
\newcommand{\R}{\mathbb R}
\newcommand{\C}{\mathbb C}

\renewcommand{\P}{\mathbb P}
\renewcommand{\O}{\mathscr O}

\renewcommand{\phi}{\varphi}
\renewcommand{\theta}{\vartheta}

\newcommand{\minus}{\setminus}
\newcommand{\inj}{\hookrightarrow}
\newcommand{\surj}{\twoheadrightarrow}

\newcommand{\isom}{\cong}
\newcommand{\mf}{\mathfrak}

\newcommand{\tensor}{\otimes}

\newcommand{\sg}{\mathrm{sg}}

\newcommand{\reg}{\mathrm{reg}}

\newcommand{\gr}{\mathrm{gr}}
\newcommand{\norm}{_\mathrm{norm}}
\newcommand{\Nef}{\mathrm{Nef}}
\newcommand{\NE}{\overline{\mathrm{NE}}}
\newcommand{\mxsc}{\mathrm{max,sc}}

\newcommand{\sE}{\mathscr{E}}
\newcommand{\sF}{\mathscr{F}}
\newcommand{\sG}{\mathscr{G}}

\newcommand{\sI}{\mathscr{I}}

\newcommand{\sL}{\mathscr{L}}

\newcommand{\sQ}{\mathscr{Q}}
\newcommand{\cQ}{\mathcal{Q}}

\newcommand{\sT}{\mathscr{T}}

\newcommand{\m}{\mathbf m}

\newdir{ ir}{{}*!/-5pt/@^{(}} 
\newdir{ il}{{}*!/-5pt/@_{(}} 

\newcommand{\eps}{\varepsilon}

\newcommand{\wt}{\widetilde}

\DeclareMathOperator{\rk}{rk}

\newenvironment{sequation}{%
\setcounter{equation}{\value{thm}}%
\numberwithin{equation}{section}%
\begin{equation}%
}{%
\end{equation}%
\numberwithin{equation}{thm}%
\addtocounter{thm}{1}%
}

\numberwithin{equation}{thm}

\DeclareRobustCommand{\SkipTocEntry}[5]{}

\setitemize[1]{leftmargin=*,parsep=0em,itemsep=0.125em,topsep=0.125em}
\setenumerate[1]{leftmargin=*,parsep=0em,itemsep=0.125em,topsep=0.125em}

\newcommand{\iref}[3]{\the\value{#1}.\the\value{#2}(\the\value{#3})}

\newcommand\factor[2]{\left. \raise 2pt\hbox{$#1$} \right/\hskip -2pt \raise -2pt\hbox{$#2$}}

\definecolor{forrest}{RGB}{81,133,49}
\definecolor{mydarkblue}{RGB}{10,92,153}

\begin{document}

\title[A Mehta--Ramanathan theorem]{A Mehta--Ramanathan theorem for linear systems with basepoints}
\author{Patrick Graf} %
\address{Lehrstuhl f\"ur Mathematik I, Universit\"at Bay\-reuth,
  95440 Bayreuth, Germany} %
\email{\href{mailto:patrick.graf@uni-bayreuth.de}{patrick.graf@uni-bayreuth.de}}
\date{October 20, 2015}
\thanks{The author was supported in full by a research grant of the
  Deutsche Forschungs\-gemeinschaft (DFG)}  %

\begin{abstract}
Let $(X, H)$ be a normal complex projective polarized variety and $\sE$ an $H$-semistable sheaf on $X$. We prove that the restriction $\sE|_C$ to a sufficiently positive general complete intersection curve $C \subset X$ passing through a prescribed finite set of points $S \subset X$ remains semistable, provided that at each $p \in S$, the variety $X$ is smooth and the factors of a Jordan--H\"older filtration of $\sE$ are locally free.
As an application, we obtain a generalization of Miyaoka's generic semipositivity theorem.
\end{abstract}

\maketitle

\section{Introduction}

In the theory of vector bundles on a normal complex projective variety~$X$, the concept of (semi-)stability is of great importance. There are various notions of stability, e.g.~slope stability and Gieseker stability.
In this paper, unless stated otherwise we always consider slope stability.
Its definition depends on the choice of a polarization, i.e.~on the choice of an ample divisor $H$ on $X$. One associates to any nonzero torsion-free coherent sheaf $\sE$ its slope
\[ \mu_H(\sE) := \frac{c_1(\sE) \cdot H^{n-1}}{\rk \sE}, \]
where $n = \dim X$. The sheaf $\sE$ is then said to be $H$-semistable if $\mu_H(\sF) \le \mu_H(\sE)$ for all proper nonzero subsheaves $\sF \subsetneq \sE$. For the notion of $H$-stability, one replaces ``$\le$'' by ``$<$'' in the above inequality.
The definition of Gieseker stability is similar, with slopes replaced by (reduced) Hilbert polynomials.

An important technical property of semistability is its invariance under restriction to general complete intersection curves $C \subset X$. More precisely, the classical theorem of Mehta and Ramanathan \cite[Thm.~6.1]{MR82} asserts that the restriction of an $H$-semistable sheaf $\sE$ on $X$ to a curve $C \subset X$ obtained as the intersection $C = D_1 \cap \cdots \cap D_{n-1}$ of general elements $D_i \in |m_i H|$ remains semistable if the $m_i$ are chosen large enough.
Actually,~\cite[Thm.~6.1]{MR82} only works if $X$ is smooth, but a stronger version valid for arbitrary normal varieties was obtained by Flenner~\cite[Thm.~1.2]{Fle84}.
These theorems sometimes allow to reduce questions about sheaves on higher-dimensional varieties $X$ to the curve or surface case by a cutting-down procedure.

In geometric applications, one might wish to concentrate attention near a prescribed (closed) point $p \in X$, i.e.~one might ask whether the restriction $\sE|_C$ to a sufficiently positive general complete intersection curve $C \subset X$ \emph{passing through $p$} is semistable.
This means that instead of the complete linear system $|m_i H|$, one only considers the subsystem consisting of divisors passing through $p$.
The Mehta--Ramanathan theorem does not provide any information about this question, as a general complete intersection curve misses the point $p$.

It is relatively easy to see that the above question has to be answered in the negative if $p$ is in a ``bad'' position relative to the sheaf $\sE$, cf.~Example~\ref{exm:example} below. This can happen even if $\sE$ is locally free at $p$.
The purpose of this paper is to show that if the point $p$ is ``in general position'', the restriction $\sE|_C$ does remain semistable. In fact, we prove a slightly more general statement where the single point $p$ is replaced by an arbitrary finite subset $S \subset X$.

The precise meaning of the general position condition in this context involves the graded object $\gr_H(\sE)$ associated to a Jordan--H\"older filtration of $\sE$. Recall that any torsion-free sheaf $\sE$ admits a filtration
\[ 0 = \sE_0 \subset \sE_1 \subset \cdots \subset \sE_\ell = \sE \]
where each quotient $\cQ_i := \factor{\sE_i}{\sE_{i-1}}$ is torsion-free and $H$-stable, and the sequence of slopes $\mu_H(\cQ_i)$ is decreasing. Such a filtration is called a Jordan--H\"older filtration of $\sE$. It is not unique, however the associated graded object $\gr_H(\sE) := \bigoplus_{i=1}^\ell \cQ_i$ is unique up to graded isomorphism and permutation of the factors.

With this notation, our main result can be stated as follows.

\begin{thm}[Mehta--Ramanathan for linear systems with basepoints] \label{thm:MR}
Let $(X, H)$ be a normal complex projective polarized variety of dimension $n \ge 2$.
Fix a torsion-free coherent sheaf $\sE$ on $X$ and a finite subset $S \subset X_\reg$, where $X_\reg$ is the smooth locus of $X$.

If $\sE$ is $H$-semistable and $\gr_H(\sE)$ is locally free at each $p \in S$, then there exist $k_0, m \in \N$ such that if $k \ge k_0$ and
\[ C = H_1 \cap \cdots \cap H_{n-1}, \quad H_i \in |kmH|, \]
is a complete intersection curve general among those passing through the set $S$, then $C$ is smooth and the restriction $\sE|_C$ is a semistable vector bundle.
\end{thm}

As a corollary, we have the following strengthening of Miyaoka's famous generic semipositivity theorem~\cite[Cor.~8.6]{Miy87}, relating the positivity of the cotangent sheaf of a normal variety $X$ to the paucity of rational curves on $X$.
By $\sT_X$ we mean the tangent sheaf of $X$, i.e.~the dual of the sheaf of K\"ahler differentials $\Omega_X^1$.

\begin{cor}[Generic semipositivity] \label{cor:gsp}
Let $X$ be a normal complex projective variety of dimension $n \ge 2$ which is not uniruled. Let $H$ be an ample divisor on $X$ and $S \subset X_\reg$ a finite set of smooth points such that $\gr_H(\sT_X)$ is locally free at each $p \in S$.
Then there exist $k_0, m \in \N$ such that if $k \ge k_0$ and
\[ C = H_1 \cap \cdots \cap H_{n-1}, \quad H_i \in |kmH|, \]
is a complete intersection curve general among those passing through the set $S$, then $C$ is smooth, contained in $X_\reg$, and $\Omega_X^1\big|_C$ is a nef vector bundle.
\end{cor}

We make some remarks about Theorem~\ref{thm:MR} and Corollary~\ref{cor:gsp}.

\begin{rem} \label{rem:377}
If $\gr_H(\sE)$ is locally free at a point $x \in X$, then so is $\sE$ itself. This follows immediately from the fact that locally near $x$, a Jordan--H\"older filtration exhibits $\sE$ as a successive extension of free sheaves.

Furthermore, $\gr_H(\sE)$ is locally free in codimension one because it is torsion-free. Hence the locus of points which in Theorem~\ref{thm:MR} are not allowed to be contained in $S$ has codimension at least two.
\end{rem}

\begin{rem}
In Corollary~\ref{cor:gsp}, if $X$ is in addition homogeneous, then $X$ is smooth and $\gr_H(\sT_X)$ is locally free.
The reason is that $g^* \gr_H(\sT_X) \isom \gr_H(\sT_X)$ for any $g \in \Aut^0(X)$, the connected component of the identity.
Hence $S \subset X$ may be any finite subset.
Note that by~\cite[Satz I]{BR62}, a homogeneous projective manifold which is not uniruled is in fact an abelian variety.
\end{rem}

\begin{rem}
Theorem~\ref{thm:MR} does not yield semistability of $\sE|_C$ for general complete intersection curves $S \subset C$ arising from \emph{all} sufficiently high multiples of $H$. The best one can get from the proof is the following statement, which is however somewhat cumbersome to formulate: for all sufficiently large $m \in \N$, there exists $k_0 = k_0(m)$, depending on $m$, such that $\sE|_C$ is semistable for $S \subset C$ arising from $|kmH|$ as soon as $k \ge k_0$.
\end{rem}

\begin{rem}
The proof of Theorem~\ref{thm:MR} only works over an uncountable algebraically closed field of characteristic zero. However, as the property of being semistable is invariant under base field extension~\cite[Thm.~1.3.7]{HL97}, the uncountability assumption is not essential. In positive characteristic, assuming $X$ is smooth we can obtain the following slightly weaker statement if we replace~\cite[Thm.~1.2]{Fle84}, which is specific to characteristic zero, by~\cite[Thm.~6.1]{MR82}:
There exist $k_0, m \in \N$ such that if $k \ge k_0$ and $S \subset C \subset X$ is a general complete intersection of divisors in $|2^kmH|$, then $\sE|_C$ is semistable.
\end{rem}

\begin{rem}
If the set $S$ contains a singular point $p \in X_\sg$, then no complete intersection curve $C \subset X$ will ever be smooth at $p$. One might still ask whether the pullback of $\sE$ to the normalization $\wt C$ is semistable.
\end{rem}

\begin{rem}
If the variety $X$ is smooth, Theorem~\ref{thm:MR} follows immediately from a result of Langer~\cite[Cor.~5.4]{Lan04}.
In fact, in this case Langer proves a much stronger, effective statement.
Hence the point of Theorem~\ref{thm:MR} really is that $X$ is allowed to be singular.
In the proof, we blow up only the prescribed finite set $S \subset X_\reg$, and we apply Flenner's restriction theorem on the blowup $\mathrm{Bl}_S X$ (which is still singular if $X$ is).
If instead we had chosen to pass to a resolution $\wt X$ of $X$ and use Langer's result on $\wt X$, the technical difficulties would have been the same.
\end{rem}

\begin{exm} \label{exm:example}
If $\gr_H(\sE)$ is not locally free at some $p \in S$, the conclusion of Theorem~\ref{thm:MR} can easily fail.
In fact, we will provide here an example of a smooth projective polarized surface $(X, H)$, a \emph{Gieseker stable} (hence slope semistable) rank $2$ vector bundle $\sE$ on $X$, and a point $p \in S$, such that for any $m > 0$ the restriction of $\sE$ to any smooth curve $C \in |mH|$ passing through $p$ is not semistable.
This gives another instance of the fact that Gieseker stability is much worse behaved than slope stability, since for a slope stable sheaf the restriction would remain stable.

Take $X$ to be a smooth projective surface of Picard number $\rho(X) \ge 2$ and with $K_X$ numerically trivial.
Pick an ample divisor $H$ on $X$ and a divisor $D$ such that $D \cdot H = 0$, but $D$ is not numerically trivial.
By Riemann--Roch and the Hodge Index Theorem, the self-intersection $D^2$ is even and negative.
Replacing $D$ by a multiple, we may assume $D^2 \le -4$.
Let $p \in X$ be any (reduced) point.
We have $H^0(X, \O_X(K_X - D)) = 0$, hence by the Serre correspondence~\cite[Thm.~5.1.1]{HL97} there exists an extension
\begin{sequation} \label{eqn:ses}
0 \lto \O_X(D) \lto \sE \lto \sI_p \lto 0,
\end{sequation}%
where $\sE$ is a rank two vector bundle and $\sI_p$ is the ideal sheaf of $p \in X$.

We will show that $\sE$ is Gieseker stable. On the other hand, $\sE$ is obviously not slope stable and the filtration given by~\eqref{eqn:ses} is a Jordan--H\"older filtration for $\sE$, so $\gr_H(\sE) = \O_X(D) \oplus \sI_p$ is not locally free at $p$.

Via Riemann--Roch, we may compute the reduced Hilbert polynomials (with respect to $H$) of the sheaves appearing in~\eqref{eqn:ses}:
\begin{align*}
p \big( \O_X(D), n \big) & = \frac 12 n^2 + \frac 1{H^2} \left( \frac 12 D^2 + \chi(\O_X) \right), \\
p(\sI_p, n)              & = \frac 12 n^2 + \frac 1{H^2} \big(         \! -1 + \chi(\O_X) \big), \\
p(\sE)                   & = \frac 12 \Big( p \big( \O_X(D) \big) + p(\sI_p) \Big).
\end{align*}
Now let $0 \neq \sF \subsetneq \sE$ be a proper nonzero subsheaf, which we may assume to be saturated. Then $\sF$ is invertible, say $\sF \isom \O_X(F)$. If the inclusion of $\sF$ factors through $\O_X(D)$, then the reduced Hilbert polynomial $p(\sF) \le p \big( \O_X(D) \big) < p(\sE)$. Otherwise, $\sF$ injects into $\sI_p$, which means that $H^0(X, \sF^* \tensor \sI_p) \ne 0$. Then $-F$ is (linearly equivalent to) a nonzero effective divisor, and $F \cdot H < 0$. Since $\frac{F \cdot H}{H^2}$ is the coefficient of the linear term in $p(\sF)$, again we have $p(\sF) < p(\sE)$. Hence $\sE$ is Gieseker stable.

It remains to show the statement about the restrictions of $\sE$.
For any $m > 0$, restriction of~\eqref{eqn:ses} to a smooth curve $C \in |mH|$ passing through $p$ yields a short exact sequence
\begin{sequation} \label{eqn:new ses}
0 \lto \O_C(D) \lto \sE|_C \lto \sI_p|_C \lto 0.
\end{sequation}%
The right-hand side sheaf has torsion at $p$ because $\sI_p$ is not locally free in $p$, cf.~\cite[Lemma~2.1.7]{HL97}. The left-hand side subsheaf is thus not saturated. Its saturation $\sG \subset \sE|_C$ satisfies
\[ \mu(\sG) > \mu(\O_C(D)) = 0 = \mu(\sE|_C), \]
whence $\sE|_C$ is not semistable.
\end{exm}

\subsection*{Outline of proof of Theorem~\ref{thm:MR}}

The natural idea for proving Theorem~\ref{thm:MR} is to pass to the blowup $f\!: \wt X \to X$ of $X$ in $S$ and apply the usual Mehta--Ramanathan theorem there. So first we have to show that the pullback $f^* \sE$ is $f^* H$-semistable. Here we already run into a slight problem, as the divisor $f^* H$ is semiample, but never ample (unless $S$ is empty). However, semistability can also be defined with respect to a semiample divisor, even if the underlying space is not \Q-factorial (see Section~\ref{sec:notations}).

Let $E$ be the exceptional divisor of $f$. As is well-known, a small perturbation of the form $f^* H - \eps E$ will be ample. Hence our goal is to show that $f^* \sE$ remains $(f^* H - \eps E)$-semistable for sufficiently small $\eps > 0$.
We establish this property in two steps. First we prove that if $\sE$ is $H$-stable, then $f^* \sE$ is $(f^* H - \eps E)$-stable for $0 < \eps \ll 1$. This is done by analyzing how the slopes of subsheaves of $f^* \sE$ depend on $\eps$. In the second step, we reduce to the stable case by considering a Jordan--H\"older filtration of $\sE$.

Once this is accomplished, we can apply the Mehta--Ramanathan theorem on $(\wt X, f^* H - \eps E)$. This yields a lot of (highly singular) complete intersection curves $S \subset C \subset X$ such that the pullback of $\sE|_C$ to the normalization $\wt C \xrightarrow\nu C$ is semistable. We would like to infer from this that $\sE|_C$ itself is semistable. Then semistability of $\sE|_{C'}$, where $C'$ is a general complete intersection curve through $S$, would follow by the openness of semistability in families~\cite[Prp.~2.3.1]{HL97}.
Here we need to consider Gieseker semistability on $C$, as it is not clear what the correct definition of slope semistability on a singular curve would be, and because~\cite[Prp.~2.3.1]{HL97} is valid only for Gieseker semistability.

Unfortunately, we are unable to show that semistability of $\nu^*(\sE|_C)$ really implies semistability of $\sE|_C$. The required condition on Hilbert polynomials is easy enough to check for \emph{locally free} subsheaves $\sF \subset \sE|_C$, however on a singular curve a saturated subsheaf of a locally free sheaf need not be locally free.
For example, consider the nodal curve singularity $C = \{ xy = 0 \}$, and let $\mf m = (x, y)$ be the ideal of the singular point. This choice of generators yields a short exact sequence
\[ 0 \lto \mf m \lto \O_C^{\oplus 2} \lto \mf m \lto 0. \]

Instead we use the numerical characterization of semistability, asserting that a degree zero vector bundle on a smooth curve is semistable if and only if it is nef.
The notion of nefness is very well-behaved under pullbacks and it also satisfies a sufficient openness property in families. This enables us to circumvent the technical difficulties outlined above.

\subsection*{Acknowledgements}

I would like to thank Thomas Peternell for suggesting working on this topic to me.
This paper has benefited from several discussions with him and with Matei Toma.
In particular, Matei Toma drew my attention to the paper~\cite{Lan04}.
Furthermore, I would like to thank the anonymous referee for helping me to improve the exposition and for asking me a question that led to a significant strengthening of Example~\ref{exm:example}.

\section{Notations and conventions} \label{sec:notations}

\subsection*{Global assumptions}

We work over the field of complex numbers $\C$.
All sheaves are assumed to be coherent.
Unless otherwise stated, divisors are assumed to have integer coefficients.

\subsection*{Projective bundles}

If $\sE$ is a vector bundle on a variety $X$, we denote by $\pi\!: \P(\sE) \to X$ the projective bundle of one-dimensional quotients of $\sE$ and by $\xi_\sE \in N^1(\P(\sE))_\Q$ the (first Chern class of the) tautological quotient bundle $\pi^*\sE \surj \O_{\P(\sE)}(1)$.

\subsection*{Stability with respect to a semiample divisor} \label{sec:stab semiample}

For a torsion-free sheaf $\sE$ on a normal projective variety $X$, the first Chern class $c_1(\sE)$ is by definition a Weil divisor class $D$ on $X$, which will in general not be \Q-Cartier. Nevertheless, the intersection number $c_1(\sE) \cdot H^{n-1} = D \cdot H^{n-1}$ can be defined for any ample (or merely semiample) divisor $H$, cf.~\cite[Sec.~1]{Mar81}. Namely, after passing to a multiple of $H$ we may assume that the linear system $|H|$ is basepoint-free. If $H_1, \dots, H_{n-1} \in |H|$ are general elements, the intersection $C = H_1 \cap \cdots \cap H_{n-1} \subset X$ is a smooth curve missing the singular locus of $X$. Hence $D$ is a Cartier divisor in a neighborhood of $C$. We may thus define $D \cdot H^{n-1}$ to be $D \cdot C$.
Alternatively, choose a resolution $f\!: \wt X \to X$ and a (Cartier) divisor $\wt D$ on $\wt X$ with $f_* \wt D = D$, e.g.~the strict transform $\wt D = f^{-1}_* D$. Then $D \cdot H^{n-1} = \wt D \cdot f^* H^{n-1}$.
These definitions are independent of all the choices made.

We can now define stability and semistability on a normal projective variety $X$ with respect to a semiample divisor $H$.
For a nonzero torsion-free coherent sheaf $\sE$ on $X$, its slope with respect to $H$ is defined by
\[ \mu_H(\sE) := \frac{c_1(\sE) \cdot H^{n-1}}{\rk \sE}. \]
The sheaf $\sE$ is said to be $H$-semistable if $\mu_H(\sF) \le \mu_H(\sE)$ for any nonzero subsheaf $\sF \subset \sE$.
The sheaf $\sE$ is said to be $H$-stable if $\mu_H(\sF) < \mu_H(\sE)$ for any nonzero subsheaf $\sF \subset \sE$ with $\rk \sF < \rk \sE$.
Note that it would be incorrect to require strict inequality also for proper subsheaves $\sF \subsetneq \sE$ of full rank.
This is because there might be a nonzero effective divisor $D \subset X$ with $D \cdot H^{n-1} = 0$. Then $\sE(-D) \subsetneq \sE$ is a full rank proper subsheaf with the same slope as $\sE$. Hence there would be no stable sheaves at all.

\section{Complete intersection curves through a fixed finite set} \label{sec:cic}

We use the notation from Theorem~\ref{thm:MR}.
Let $m$ be a positive integer such that $mH$ is very ample. Set
\[ V_m = \{ s \in H^0(X, mH) \;|\; s(P) = 0 \text{ for all $P \in S$} \} \]
and let $S_m = \P(V_m)$ be the projective space of lines in $V_m$.
If $\m = (m_1, \dots, m_{n-1})$ is a tuple of positive integers such that each $m_i H$ is very ample, set $S_\m = S_{m_1} \times \cdots \times S_{m_{n-1}}$ and let $Z_\m \subset X \times S_\m$ be the incidence correspondence
\[ Z_\m = \{ (x, s_1, \dots, s_{n-1}) \in X \times S_\m \;|\; s_i(x) = 0 \text{ for all $1 \le i \le n - 1$} \}. \]
We obtain a diagram
\[ \xymatrix{
X \times S_\m & Z_\m \ar@{ il->}[l] \ar^{p_\m}[r] \ar_{q_\m}[d] & X \\
& S_\m & 
} \]
with $p_\m$ and $q_\m$ the projections. The fibre of $q_\m$ over a point $(s_1, \dots, s_{n-1}) \in S_\m$ is (isomorphic to) the subscheme of $X$ cut out by the equations $s_1 = \cdots = s_{n-1} = 0$.

Let $X^\circ \subset X$ be the open subset of $X$ where $X$ is smooth and $\sE$ is locally free. We define two subsets $U_\m' \subset U_\m \subset S_\m$ by
\begin{align*}
U_\m & = \{ s \in S_\m \;|\; \dim q_\m^{-1}(s) = 1 \text{ and } q_\m^{-1}(s) \subset X^\circ \}, \\
U_\m' & = \{ s \in U_\m \;|\; q_\m^{-1}(s) \text{ is smooth} \}.
\end{align*}

\begin{prp}[Properties of $U_\m$ and $U_\m'$] \label{prp:qm}
\quad
\begin{enumerate}
\item\label{itm:qm.1} The subset $U_\m \subset S_\m$ is open and dense. The morphism $q_\m^{-1}(U_\m) \to U_\m$ is flat with connected fibres, and $p_\m^* \sE|_{q_\m^{-1}(U_\m)}$ is locally free.
In particular, $p_\m^* \sE|_{q_\m^{-1}(U_\m)}$ is flat over $U_\m$.
\item\label{itm:qm.2} The subset $U_\m' \subset U_\m$ is open.
If $\wt H_{m_i} := f^*(m_i H) - E$ is very ample on $\wt X$ for each $i$,
then $U_\m'$ is nonempty.
\end{enumerate}
\end{prp}

\begin{proof}
For~(\ref{prp:qm}.\ref{itm:qm.1}), the set $\{ s \in S_\m \;|\; \dim q_\m^{-1}(s) = 1 \}$ is open by~\cite[Cor.~13.1.5]{EGA_IV_3}. The set $\{ s \in S_\m \;|\; q_\m^{-1}(s) \subset X^\circ \}$ is open since its complement is $q_\m(p_\m^{-1}(X \minus X^\circ))$, which is closed as $q_\m$ is proper. Hence $U_\m$ is open. If $s \in S_\m$ is general, then clearly $\dim q_\m^{-1}(s) = 1$, and $q_\m^{-1}(s) \subset X^\circ$ because $S \subset X^\circ$ and the complement of $X^\circ$ has codimension at least two in $X$. So $U_\m$ is nonempty.

It follows by the same argument as in the proof of~\cite[Prp.~1.5.i)]{MR82} that the morphism $q_\m^{-1}(U_\m) \to U_\m$ is flat. Its fibres, being intersections of ample divisors, are connected by~\cite[Ch.~III, Cor.~7.9]{Har77}. As $p_\m$ maps $q_\m^{-1}(U_\m)$ to $X^\circ$ by definition, the sheaf $p_\m^* \sE|_{q_\m^{-1}(U_\m)}$ is locally free.

Concerning~(\ref{prp:qm}.\ref{itm:qm.2}), the set $U_\m'$ is open by~\cite[Thm.~12.2.4.iii)]{EGA_IV_3}. Let $f\!: \wt X \to X$ be the blowup of $X$ in $S$. For any $p \in S$, let $E_p \subset \wt X$ be the exceptional divisor over $p$, and let $E = \sum E_p$ be the total exceptional divisor.
If $\wt H_{m_i}$ is very ample for each $i$, then a general curve $\wt C \in |\wt H_{m_1} \cap \cdots \cap \wt H_{m_{n-1}}|$ is smooth by Bertini's theorem. By an intersection number computation,
\[ \wt H_{m_1} \cdots \wt H_{m_{n-1}} \cdot E_p = (-1)^{n-1} E_p^n = 1 \quad\text{for all $p \in S$.} \]
This means that $\wt C$ intersects each divisor $E_p$ transversally in a single point. It follows that $\wt C$ maps isomorphically onto its image $C = f(\wt C) \subset X$, which is therefore smooth.
Under the isomorphism
\[ \P H^0(\wt X, \wt H_m) \isom S_m \]
the curve $\wt C$ corresponds to a point $s \in U_\m$ with $q_\m^{-1}(s) = C$, showing that $U_\m' \ni s$ is nonempty.
\end{proof}

\section{\Q-twisted vector bundles}

We will use the formalism of \Q-twisted vector bundles as presented in \cite[Sec.~6.2]{Laz04b}. For the reader's convenience, we recall here the notation and the most important facts.

\begin{dfn}[\Q-twisted bundles]
A \Q-twisted vector bundle $\sE\tw\delta.$ on a variety $X$ is an ordered pair $(\sE, \delta)$, where $\sE$ is a vector bundle on $X$ and $\delta \in N^1(X)_\Q$ is a numerical equivalence class.
The pullback of $\sE\tw\delta.$ by a morphism $f\!: Y \to X$ is defined as
\[ f^*( \sE\tw\delta. ) := (f^* \sE)\tw f^* \delta. . \]
\end{dfn}

\begin{dfn}[Normalized bundles]
Let $\sE$ be a vector bundle of rank $r$ on $X$. The normalized bundle $\sE\norm$ is defined to be the \Q-twisted bundle
\[ \sE\norm := \sE \big\langle \!-\textstyle\frac{1}{r} c_1(\sE) \big\rangle. \]
\end{dfn}

It is clear that normalization commutes with pullback, i.e.~for arbitrary morphisms $f\!: Y \to X$ we have
\[ f^*(\sE\norm) = (f^* \sE)\norm. \]

\begin{dfn}[Nef bundles]
Assume that $X$ is projective. A \Q-twisted bundle $\sE\tw\delta.$ is said to be \emph{nef} if
\[ \xi_\sE + \pi^* \delta \in N^1(\P(\sE))_\Q \]
is nef, where $\pi\!: \P(\sE) \to X$ is the bundle map.
\end{dfn}

\begin{lem}[Nefness and pullbacks] \label{lem:nef pullback}
Let $f\!: Y \to X$ be a surjective morphism of projective varieties. Then a \Q-twisted bundle $\sE\tw\delta.$ on $X$ is nef if and only if $f^*(\sE\tw\delta.)$ is nef.
\end{lem}

\begin{proof}
Consider the commutative diagram
\[ \xymatrix{
\P(f^* \sE) \ar_{\pi_Y}[d] \ar^g[r] & \P(\sE) \ar^{\pi_X}[d] \\
Y \ar^f[r] & X.
} \]
Clearly $\xi_{f^* \sE} + \pi_Y^*(f^* \delta) = g^*(\xi_\sE + \pi_X^* \delta)$. Hence the statement reduces to the case of line bundles, where it is well-known~\cite[Ex.~1.4.4]{Laz04a}.
\end{proof}

\begin{prp}[Criterion for semistability] \label{prp:ss crit}
Let $C$ be a smooth projective curve. Then a vector bundle $\sE$ on $C$ is semistable if and only if $\sE\norm$ is nef.
\end{prp}

\begin{proof}
See~\cite[Prp.~6.4.11]{Laz04b}.
\end{proof}

\begin{rem} \label{rem:stable?}
It would be interesting to have a similar numerical criterion for stability. However, it seems that no such criterion can exist. Consider for simplicity the case of a semistable rank two, degree zero vector bundle $\sE$ on the curve $C$. In this case, $\sE\norm = \sE$ and $\pi\!: X = \P(\sE) \to C$ is a ruled surface.
It is easy to see (cf.~\cite[Sec.~1.5.A]{Laz04a}) that $\sE$ is stable if and only if
\[ H^0(X, L) = 0 \quad\text{for all $L \in \Pic(X)$ with $c_1(L) = \xi_\sE$.} \]
This is clearly not a numerical condition. To be more precise, whether $\sE$ is stable or not, in the N\'eron--Severi space of $X$ we have
\[ \Nef(X) = \NE(X) = \R^{\ge 0} \cdot \xi_\sE + \R^{\ge 0} \cdot f. \]
Here $f$ is the class of a fibre of $\pi$. Hence the cases $\sE$ strictly semistable (i.e.~semistable but not stable) and $\sE$ stable are indistinguishable from a numerical point of view.
\end{rem}

\section{Semistability of pullbacks}

We consider the setup from Theorem~\ref{thm:MR}.
Let $f\!: \wt X \to X$ be the blowup of $X$ in $S$. For any $p \in S$, let $E_p \subset \wt X$ be the exceptional divisor over $p$, and let $E = \sum E_p$ be the total exceptional divisor.
For brevity, we denote $f^* H - \eps E$ by $H_\eps$.
The aim of the present section is to establish the following property.

\begin{prp}[Semistability of pullbacks] \label{prp:pullback ss}
For $0 < \eps \ll 1$, the divisor $H_\eps$ is ample and the sheaf $f^* \sE$ is $H_\eps$-semistable.
\end{prp}

The analogue of Proposition~\ref{prp:pullback ss} for stability will serve as the start of induction in the proof.

\begin{prp}[Stability of pullbacks] \label{prp:pullback stable}
If $\sE$ is $H$-stable, then for $0 < \eps \ll 1$ the divisor $H_\eps$ is ample and the sheaf $f^* \sE$ is $H_\eps$-stable.
\end{prp}

\begin{rem}
If we assume $X$ to be \Q-factorial, Proposition~\ref{prp:pullback stable} follows immediately from~\cite[Thm.~3.3]{GKP14}. However, from the point of view of Theorem~\ref{thm:MR} a \Q-factoriality assumption seems quite unnatural.
\end{rem}

\subsection{Suprema of slopes of subsheaves}

In the proof of Proposition~\ref{prp:pullback stable}, we need to consider the following quantity associated to a torsion-free sheaf.

\begin{dfn}[Supremum of slopes of strict subsheaves] \label{dfn:mu max}
Let $\sE$ be a nonzero torsion-free sheaf on the normal projective variety $X$, and let $H$ be a semiample divisor on $X$. We define
\[ \mu_H^\mxsc(\sE) := \sup \{ \mu_H(\sF) \;|\; \sF \subset \sE \text{ a subsheaf with } 0 < \rk \sF < \rk \sE \}. \]
\end{dfn}

\begin{prp}[Supremum is attained] \label{prp:mu max}
The supremum in Definition~\ref{dfn:mu max} is in fact a maximum. In particular, we have $\mu_H^\mxsc(\sE) < \infty$, and if $\sE$ is $H$-stable, then $\mu_H^\mxsc(\sE) < \mu_H(\sE)$.
\end{prp}

\begin{proof}
By Serre vanishing, for some sufficiently ample line bundle $\sL$ the sheaf $\sE^* \tensor \sL$ is globally generated, i.e.~it is a quotient of a free sheaf $\O_X^{\oplus m}$. Twisting by $\sL^*$ and dualizing, we see that the double dual $\sE^{**}$ can be embedded into the direct sum $\sL^{\oplus m}$. As $\sE$ is torsion-free, the natural map $\sE \to \sE^{**}$ is injective and hence also $\sE$ itself is embedded in $\sL^{\oplus m}$. It follows that the set of slopes in question is bounded above.
Being contained in the set $\frac{1}{(\rk \sE)!} \cdot \Z$, it does not have an accumulation point in $\R$. Thus it has a maximum.
\end{proof}

\subsection{Auxiliary lemmas}

We switch back to the notation introduced at the beginning of this section.
The following two lemmas will be used later.

\begin{lem}[Weak semistability of pullbacks] \label{lem:pullback ss}
Let $\sF$ be a torsion-free sheaf on $X$ which is locally free at each $p \in S$. If $\sF$ is $H$-semistable (resp., $H$-stable), then $f^* \sF$ is $f^* H$-semistable (resp., $f^* H$-stable).
\end{lem}

\begin{proof}
We only deal with the semistable case, as the stable case is similar.
It is clear that $f^* \sF$ is a torsion-free sheaf on $\wt X$ with $\mu_{f^* H}(f^* \sF) = \mu_H(\sF)$. Let $\sG \subset f^* \sF$ be any nonzero subsheaf. Pushing down, we get an inclusion $f_* \sG \subset f_* f^* \sF = \sF$. Hence $\mu_H(f_* \sG) \le \mu_H(\sF)$. But $\mu_{f^* H}(\sG) = \mu_H(f_* \sG)$, so $\mu_{f^* H}(\sG) \le \mu_{f^* H}(f^* \sF)$.
\end{proof}

\begin{lem}[Extensions of semistable sheaves] \label{lem:ext ss}
Let
\[ 0 \lto \sF' \lto \sF \lto \sF'' \lto 0 \]
be a short exact sequence of torsion-free sheaves on $X$. If $\sF'$ and $\sF''$ are $H$-semistable of the same slope $\mu$, then also $\sF$ is $H$-semistable of slope $\mu$.
\end{lem}

\begin{proof}
It is easily verified that $\mu_H(\sF) = \mu$. Proceeding by contradiction, assume that $\sF$ is not $H$-semistable, and let $0 \ne \sG \subset \sF$ be the maximally destabilizing subsheaf~\cite[Def.~1.3.6]{HL97}. Then $\sG$ is $H$-semistable of slope $\mu_H(\sG) > \mu$. By~\cite[Prp.~1.2.7]{HL97}, the induced map $\sG \to \sF''$ is zero. Hence $\sG \inj \sF$ factors through a map $\sG \to \sF'$, which is zero for the same reason. It follows that $\sG = 0$, a contradiction.
\end{proof}

\subsection{Proof of Proposition~\ref{prp:pullback stable}}

Replacing $E$ by a sufficiently small positive multiple, we may assume that $f^* H - E$ is ample. We denote the slope with respect to $H_\eps = f^* H - \eps E$ by $\mu_\eps$.

\begin{lem}[Dependence of slopes on perturbation] \label{lem:mu eps}
For any torsion-free sheaf $\sF$ on $\wt X$, there is a constant $C = C(\sF)$ such that
\[ \mu_\eps(\sF) = \mu_0(\sF) + C \cdot \eps^{n-1}. \]
If $\sF$ is free in a neighborhood of $E$, then $C(\sF) = 0$.
\end{lem}

\begin{proof}
Let $g\!: Y \to \wt X$ be a resolution of singularities, and let $D_Y$ be the strict transform of a Weil divisor $D$ representing $c_1(\sF)$. Then, noting that $f^* H \cdot E = 0$, we calculate
\begin{equation*}
\begin{array}{ccl}
c_1(\sF) \cdot H_\eps^{n-1} & = & D_Y \cdot g^* H_\eps^{n-1} \\
  & = & D_Y \cdot g^* f^* H^{n-1} + \eps^{n-1} \cdot D_Y \cdot (-g^* E)^{n-1}.
\end{array}
\end{equation*}
The first claim follows. If $\sF$ is free in a neighborhood of $E$, then $D$ can be chosen disjoint from $E$, whence $D_Y \cdot (-g^* E)^{n-1} = 0$.
\end{proof}

Now consider the function
\[ \Phi(\eps) := \mu_0^\mxsc(f^* \sE) + \eps^{n-1} \cdot \big( \mu_1^\mxsc(f^* \sE) - \mu_0^\mxsc(f^* \sE) \big). \]
We have $\Phi(0) = \mu_0^\mxsc(f^* \sE) < \mu_0(f^* \sE)$ by Lemma~\ref{lem:pullback ss} and Proposition~\ref{prp:mu max}. By continuity, there exists $\eps_0 > 0$ such that $\Phi(\eps) < \mu_\eps(f^* \sE)$ for any $0 < \eps < \eps_0$.
We claim that for any subsheaf $\sF \subset f^* \sE$ with $0 < \rk \sF < \rk \sE$, we have $\mu_\eps(\sF) \le \Phi(\eps)$ for $0 \le \eps \le 1$. This implies that $f^* \sE$ is $H_\eps$-stable for $\eps < \eps_0$.
Note that
\begin{equation*}
\begin{array}{l}
\mu_0(\sF) \le \mu_0^\mxsc(f^* \sE) = \Phi(0) \quad \text{and} \\[1ex]
\mu_1(\sF) \le \mu_1^\mxsc(f^* \sE) = \Phi(1).
\end{array}
\end{equation*}
The claim is thus a consequence of Lemma~\ref{lem:mu eps} and the following elementary assertion. \qed

\begin{lem}
Let $f, g$ be two real polynomials of the form $a_0 + a_k x^k$, for the same $k \ge 1$. If $f(0) \le g(0)$ and $f(1) \le g(1)$, then $f \le g$ on the interval $[0, 1]$.
\end{lem}

\begin{proof}
It suffices to show that if $h = b_0 + b_k x^k$ is a polynomial of the above form, then $h(0) \le 0$ and $h(1) \le 0$ imply $h \le 0$ on $[0, 1]$.
If $h(x) > 0$ for some $x \in (0, 1)$, then $h|_{[0, 1]}$ attains its maximum at some $x_0 \in (0, 1)$, and then $h'(x_0) = k b_k x_0^{k-1} = 0$. Thus $b_k = 0$, i.e.~$h$ is constant and the assertion is clear.
\end{proof}

\subsection{Proof of Proposition~\ref{prp:pullback ss}}

The proof proceeds by induction on the length $\ell$ of a Jordan--H\"older filtration of $\sE$ with respect to the polarization $H$. If $\ell = 1$, i.e.~if $\sE$ is $H$-stable, then by Proposition~\ref{prp:pullback stable} for $0 < \eps \ll 1$ the sheaf $f^* \sE$ is $H_\eps$-stable, in particular $H_\eps$-semistable.

So let $\ell > 1$, and assume that the claim has already been shown for sheaves admitting a Jordan--H\"older filtration of length at most $\ell - 1$. Let
\begin{sequation} \label{eqn:JHF E}
0 = \sE_0 \subset \sE_1 \subset \cdots \subset \sE_\ell = \sE
\end{sequation}%
be a Jordan--H\"older filtration of $\sE$. This gives rise to a short exact sequence
\begin{sequation} \label{eqn:E_1}
0 \lto \sE_1 \lto \sE \lto \sG \lto 0.
\end{sequation}%

\begin{clm} \label{clm:JHF}
The following properties hold true.
\begin{enumerate}
\item\label{itm:JHF.1} $\sE_1$ is torsion-free and $H$-stable of slope $\mu := \mu_H(\sE)$.
\item\label{itm:JHF.2} $\sG$ is torsion-free and $H$-semistable of slope $\mu$.
\item\label{itm:JHF.3} $\sG$ has a Jordan--H\"older filtration of length exactly $\ell - 1$.
\item\label{itm:JHF.4} $\sE_1$, $\sG$ and $\gr_H(\sG)$ are locally free at each $p \in S$.
\end{enumerate}
In particular, we may apply the inductive hypothesis to $\sE_1$ and to $\sG$.
\end{clm}

\begin{proof}
Claim~(\ref{clm:JHF}.\ref{itm:JHF.1}) and the first part of~(\ref{clm:JHF}.\ref{itm:JHF.2})
are immediate from the definition of a Jordan--H\"older filtration.
The filtration~\eqref{eqn:JHF E} induces a Jordan--H\"older filtration $\{ \sE_i / \sE_1 \}_{1 \le i \le \ell}$ of $\sG$ which exhibits $\sG$ as a successive extension of $H$-stable sheaves of slope $\mu$, hence $\sG$ is $H$-semistable of slope $\mu$ by Lemma~\ref{lem:ext ss}.
As the filtration $\{ \sE_i / \sE_1 \}$ has length $\ell - 1$, Claim~(\ref{clm:JHF}.\ref{itm:JHF.3}) is obvious. Finally, we have
\[ \gr_H(\sE) = \gr_H(\sG) \oplus \sE_1 \]
by definition, so local freeness of $\gr_H(\sE)$ at $p \in S$ implies that the two summands on the right-hand side enjoy the same property.
As already mentioned in Remark~\ref{rem:377}, the sheaf $\sG$ is locally free wherever $\gr_H(\sG)$ is locally free.
\end{proof}

Pulling back sequence~\eqref{eqn:E_1} to $\wt X$, we obtain
\begin{sequation} \label{eqn:f^* E_1}
0 \lto f^* \sE_1 \lto f^* \sE \lto f^* \sG \lto 0.
\end{sequation}%
By Lemma~\ref{lem:pullback ss}, the sheaves $f^* \sE_1$ and $f^* \sG$ are $f^* H$-semistable. Hence by the inductive assumption, they are also $H_\eps$-semistable for sufficiently small $\eps > 0$. Furthermore, as $\sE_1$ is locally free at each $p \in S$, we have
\[ \mu_\eps(f^* \sE_1) = \mu_0(f^* \sE_1) = \mu_H(\sE_1) = \mu \]
by Lemma~\ref{lem:mu eps}. By the same reasoning, $\mu_\eps(f^* \sG) = \mu$. Sequence~\eqref{eqn:f^* E_1} thus exhibits $f^* \sE$ as an extension of $H_\eps$-semistable sheaves of the same slope. By Lemma~\ref{lem:ext ss}, we obtain that $f^* \sE$ itself is $H_\eps$-semistable, which was to be shown. \qed

\section{Semistability in families} \label{sec:ss fam}

Let $f\!: X \to B$ be a flat projective morphism with connected fibres such that the general fibre is a smooth curve.
Let $\sE\tw\delta.$ be a \Q-twisted vector bundle on $X$. For $t \in B$, we set $X_t := f^{-1}(t)$ and we let $\sE\tw\delta._t$ be the restriction of $\sE\tw\delta.$ to $X_t$.
In this section, we prove the following result.

\begin{prp}[Semistability in families] \label{prp:ss families}
Assume that for some $t_0 \in B$, the fibre $X_{t_0}$ is reduced and irreducible with normalization $\nu\!: \wt X_{t_0} \to X_{t_0}$, and that $\nu^*(\sE_{t_0})$ is semistable. Then the bundle $\sE_t$ is semistable for general $t \in B$.
\end{prp}

\begin{lem}[Weak openness of nefness] \label{lem:nefness families}
Assume that for some $t_0 \in B$, the restriction $\sE\tw\delta._{t_0}$ is nef. Then $\sE\tw\delta._t$ is nef for very general $t \in B$, i.e.~for $t$ outside a countable union $\bigcup B_i$ of proper subvarieties $B_i \subsetneq B$.
\end{lem}

\begin{proof}
Consider the projectivized bundle $\pi\!: \P_X(\sE) \to X$, and let $h = f \circ \pi\!: \P_X(\sE) \to B$ be the induced map. For any $t \in B$, the bundle $\sE\tw\delta._t$ is nef if and only if $(\xi_\sE + \pi^* \delta)|_{h^{-1}(t)}$ is nef. Thus the statement reduces to the case of line bundles, which is~\cite[Prp.~1.4.14]{Laz04a}.
\end{proof}

\begin{rem}
It is not known whether the locus of fibres where a given line bundle fails to be nef can really be an infinite union of subvarieties not contained in each other.
Moriwaki~\cite[Exm.~7]{Mor92} gave an example of such behavior in characteristic $p > 0$. Lesieutre~\cite[Thm.~1.2]{Les14} showed that this is also possible over the complex numbers if one allows $\R$-divisors instead of line bundles.

Over a countable field of positive characteristic, it may happen that a line bundle is nef over the generic geometric point, but not nef over any closed geometric point~\cite[Exm.~5.3]{Lan13}.
By~\cite[Sec.~8]{Lan15}, the same phenomenon can also occur in mixed characteristic, e.g.~over $\Spec \Z[1/N]$.
\end{rem}

\begin{proof}[Proof of Proposition~\ref{prp:ss families}]
Since $\nu^* \sE_{t_0}$ is semistable, $(\nu^* \sE_{t_0})\norm$ is nef by Proposition~\ref{prp:ss crit}. But
\[ (\nu^* \sE_{t_0})\norm = \nu^* \big( (\sE\norm)_{t_0} \big), \]
hence by Lemma~\ref{lem:nef pullback}, also $(\sE\norm)_{t_0}$ is nef. By Proposition~\ref{lem:nefness families}, $(\sE\norm)_t = (\sE_t)\norm$ is nef for very general $t \in B$. Choose such a $t_1 \in B$ where additionally $X_{t_1}$ is a smooth curve. Then $\sE_{t_1}$ is semistable by Proposition~\ref{prp:ss crit} again. The set $\{ t \in B \;|\; \sE_t \text{ is semistable} \}$ is open~\cite[Prp.~2.3.1]{HL97}, so the claim follows.
\end{proof}

\section{Proof of Theorem~\ref{thm:MR}}

Let $f\!: \wt X \to X$ be the blowup of $X$ in $S$. As before, for any point $p \in S$ let $E_p \subset \wt X$ be the exceptional divisor over $p$, and let $E = \sum E_p$ be the total exceptional divisor.
By~\cite[Ch.~II, Ex.~7.14.b)]{Har77} and Proposition~\ref{prp:pullback ss}, for a sufficiently large integer $m \gg 0$ the divisor $\wt H_m := f^*(mH) - E$ is very ample on $\wt X$ and $f^* \sE$ is $\wt H_m$-semistable.
By Flenner's Mehta--Ramanathan theorem~\cite[Thm.~1.2]{Fle84}, for $k \in \N$ sufficiently large the sheaf $f^* \sE|_{\wt C}$ is semistable, where
\[ \wt C = H_1 \cap \cdots \cap H_{n-1}, \quad H_i \in |k \wt H_m|, \]
is a general complete intersection curve. The image curve $C = f(\wt C)$ is then a (singular) curve arising as a complete intersection of divisors in $|kmH|$.

The curve $\wt C$ meets, but is not contained in, each of the divisors $E_p$. By codimension reasons it misses the singular locus of $\wt X$ and the locus where $f^* \sE$ is not locally free.
Hence $C$ passes through $S$ and is contained in the set $X^\circ$ defined in Section~\ref{sec:cic}. Using notation from that Section, we thus obtain a point $s \in U_\m$ with $q_\m^{-1}(s) = C$, where $\m = (km, \dots, km)$.
The map $f|_{\wt C}\!: \wt C \to C$ is the normalization of $C$, and $(f|_{\wt C})^* (\sE|_C) = f^* \sE|_{\wt C}$ is semistable.
Note that by Proposition~\ref{prp:qm}.\ref{itm:qm.2}, the set $U_\m'$ is nonempty, i.e.~the general member of the family of curves $q_\m^{-1}(U_\m) \to U_\m$ is smooth.
Theorem~\ref{thm:MR} now follows by applying Proposition~\ref{prp:ss families} to the family $q_\m^{-1}(U_\m) \to U_\m$ and the sheaf $p_\m^* \sE$ on $q_\m^{-1}(U_\m)$.

\section{Proof of Corollary~\ref{cor:gsp}}

Consider the Harder--Narasimhan (HN) filtration of $\sT_X$ with respect to the polarization $H$,
\begin{sequation} \label{eqn:HNF}
0 = \sF_0 \subset \sF_1 \subset \dots \subset \sF_r = \sT_X.
\end{sequation}%
The defining property of the HN filtration is that the quotients $\sQ_i = \sF_i/\sF_{i-1}$ are torsion-free and $H$-semistable, and $\mu_H(\sQ_i) > \mu_H(\sQ_{i+1})$ for all indices $i$.
For existence and uniqueness of the HN filtration, see~\cite[Thm.~1.3.4]{HL97}.

As $\gr_H(\sT_X) = \bigoplus_i \gr_H(\sQ_i)$, local freeness of $\gr_H(\sT_X)$ at each $p \in S$ implies the same property for each $\gr_H(\sQ_i)$. Hence we may apply Theorem~\ref{thm:MR} to each of the finitely many sheaves $\sQ_i$. We obtain numbers $k_0, m \in \N$ such that if $k \ge k_0$ and
\[ C = H_1 \cap \cdots \cap H_{n-1}, \quad H_i \in |kmH|, \]
is a general complete intersection curve passing through $S$, then $C$ is smooth, contained in $X_\reg$, and each $\sQ_i|_C$ is semistable. We will show that this implies nefness of $\Omega_X^1|_C$, ending the proof.

Arguing by contradiction, suppose that $\Omega_X^1|_C$ is not nef. Then there is a locally free quotient $\Omega_X^1|_C \surj \sG$ of negative degree~\cite[Thm.~6.4.15]{Laz04b}.
Dualizing, we get a subbundle $\sF \subset \sT_X|_C$ of positive degree. In particular, the first term $\sF_1^C$ of the HN filtration
\[ 0 = \sF_0^C \subset \sF_1^C \subset \dots \subset \sF_s^C = \sT_X|_C \]
of $\sT_X|_C$ has positive degree.
Now note that by semistability of the $\sQ_i|_C$, the restriction $\{ \sF_i|_C \}$ of filtration~\eqref{eqn:HNF} to $C$ has semistable quotients with strictly decreasing slopes. Hence $\{ \sF_i|_C \}$ is the HN filtration of $\sT_X|_C$, i.e.~$r = s$ and $\sF_i|_C = \sF_i^C$ for all $0 \le i \le r$. So the subsheaf $\sF_1 \subset \sT_X$ satisfies $c_1(\sF_1) \cdot H^{n-1} = (km)^{-(n-1)} \deg \sF_1^C > 0$. By~\cite[Thm.~9.0.2]{Kol92}, $X$ is uniruled, leading to the desired contradiction.

\newcommand{\etalchar}[1]{$^{#1}$}
\providecommand{\bysame}{\leavevmode\hbox to3em{\hrulefill}\thinspace}
\providecommand{\MR}{\relax\ifhmode\unskip\space\fi MR}
\providecommand{\MRhref}[2]{%
  \href{http://www.ams.org/mathscinet-getitem?mr=#1}{#2}
}
\providecommand{\href}[2]{#2}

\end{document}